\pdfoutput 1 
\documentclass[12pt]{article}

\usepackage{amssymb,amsmath,amsthm}
\usepackage{setspace}
\usepackage{fullpage}
\usepackage{graphicx}

\newtheorem{theorem}{Theorem}
\newtheorem{lemma}{Lemma}
\newtheorem{proposition}{Proposition}

\theoremstyle{definition}
\newtheorem{definition}{Definition}

\newcommand{\bigS}{\mathcal{S}}

\newcommand{\M}{\mathcal{M}}
\newcommand{\E}{\textbf{E}}

\DeclareMathOperator{\per}{Per}
\DeclareMathOperator{\sel}{Sel}

\title{A recursive construction of $t$-wise uniform permutations}
\author{Hilary Finucane\footnote{Weizmann Institute of Science, Israel. Email: hilary.finucane@weizmann.ac.il. Supported by an ERC grant.} \and Ron Peled\footnote{Tel Aviv University, Israel. E-mail: peledron@post.tau.ac.il. Supported by an ISF grant and an IRG grant.} \and Yariv Yaari\footnote{Weizmann Institute of Science, Israel.}}

\begin{document}
\maketitle

\begin{abstract}
We present a recursive construction of a $(2t+1)$-wise uniform set of permutations on $2n$ objects using a $(2t+1)-(2n,n,\cdot)$ combinatorial design, a $t$-wise uniform set of permutations on $n$ objects and a $(2t+1)$-wise uniform set of permutations on $n$ objects. Using the complete design in this procedure gives a $t$-wise uniform set of permutations on $n$ objects whose size is at most $t^{2n}$, the first non-trivial construction of an infinite family of $t$-wise uniform sets for $t \geq 4$. If a non-trivial design with suitable parameters is found, it will imply a corresponding improvement in the construction.
\end{abstract}


\section{Introduction}
A $t$-wise uniform set of permutations is a subset of the symmetric
group $S_n$ which has the same statistics on any $t$-tuple as $S_n$. In other words:

\begin{definition}
  A \emph{$t$-wise uniform set} (of permutations on $n$ objects) is a subset $T\subseteq S_n$ such that for any distinct $i_1,\ldots, i_t\in[n]$ and any distinct
  $j_1,\ldots, j_t\in[n]$, we have that the probability that $\sigma(i_m) = j_m$ for all $1 \leq m \leq t$ is the same whether $\sigma$ is chosen uniformly from $S_n$ or uniformly from $T$.
  \end{definition}
Equivalently, $T$ is a $t$-wise uniform set if
  \begin{equation}
  \label{t-wise_def}
    \frac{1}{|T|}|\{\sigma\in T\,: \, \sigma(i_m)=j_m\text{ for $1\leq m \leq t$}\}| = \frac{1}{n(n-1)\cdots(n-t+1)}
  \end{equation}
  for every distinct $i_1,\ldots, i_t\in[n]$ and every distinct
  $j_1, \ldots, j_t\in[n]$.

There are two \emph{trivial} constructions of $t$-wise uniform sets:
The symmetric group $S_n$ is a $t$-wise uniform set for any $t \leq
n$, and the alternating group $A_n$ is a $t$-wise uniform set for $t
\leq n-1$ (and $n>2$). In this paper we consider the problem of
constructing non-trivial $t$-wise uniform sets.

The problem of finding explicit constructions for $t$-wise uniform
sets was posed as an open problem by Kaplan, Naor, and Reingold in
\cite{KaplanNaorReingold05}. It was also shown there that
\emph{approximate} $t$-wise uniform sets of small size exist.
Approximate $t$-wise uniform sets were further explored in
\cite{AlonLovett12}, where Alon and Lovett showed that there exists
a perfect $t$-wise uniform \emph{distribution} over any approximate
$t$-wise uniform set, a useful result for derandomization.

Non-trivial explicit constructions of (non-approximate) $t$-wise uniform sets for
infinitely many $n$ are known only for $t=1,2,3$: the group of
cyclic shifts $x \mapsto x+a$ modulo $n$ is a $1$-wise uniform set,
the group of invertible affine transformations $x \mapsto ax+b$ over
a finite field $\mathbb{F}$ yields a $2$-wise uniform set, and the
group of M\"{o}bius transformations $x \mapsto (ax+b)/(cx+d)$ with
$ad-bc=1$ over the projective line $\mathbb{F} \cup \{\infty\}$
yields a $3$-wise uniform set. Moreover, it is known (see for
example \cite[Theorem 5.2]{Cameron95}) that for $n \ge 25$ and $t
\ge 4$ there are no \emph{subgroups} of $S_n$, other than $A_n$ and
$S_n$ itself, that form a $t$-wise uniform set; such subgroups are
called $t$-transitive subgroups of $S_n$. In contrast, it was shown
recently \cite{KuperbergLovettPeled11} that for all $n \geq 1$ and
$1 \leq t \leq n$, there exists a $t$-wise uniform set of
permutations on $n$ letters of size $n^{ct}$ for some universal
constant $c>0$. For small $t$ and large $n$, this result is close to
the simple lower bound of $n(n-1)\cdots(n-t+1)$ which is implied by
\eqref{t-wise_def}. It is important to emphasize, however, that the
proof in \cite{KuperbergLovettPeled11} is purely existential and
provides no hint as to the construction of such small $t$-wise
uniform sets. Our work gives the first non-trivial explicit
construction of an infinite family of $t$-wise uniform sets for $t
\geq 4$.

A natural approach to constructing $t$-wise uniform sets is the divide-and-conquer method. To choose a permutation $\sigma$ on $2n$ letters, it suffices to do the following.
\begin{description}
\item[Step 1.] Choose the set $S\subseteq [2n]$ of $n$ elements that will be mapped by $\sigma$ to $1, \ldots, n$.
\item[Step 2.] Choose two permutations $\tau_1$ and $\tau_2$ on $n$ letters each to determine the behavior of $\sigma$ on $S$ and $S^c$.
\end{description}
If both steps are done independently and uniformly at random, then
the resulting permutation $\sigma$ is uniformly random. Moreover, if
Step 1 is done uniformly at random and $\tau_1$ and $\tau_2$ are
sampled independently (of each other and of $S$) and uniformly
from two $t$-wise uniform sets, then the resulting family of
permutations forms a $t$-wise uniform set. Indeed, in this case if
$0\le m\le t$, $1 \leq j_1, \ldots , j_m \leq n$ and $n+1 \leq
j_{m+1}, \ldots, j_t \leq 2n$ are distinct indices and $1 \leq i_1,
\ldots , i_t \leq 2n$ are distinct indices, then
\begin{align*}
\Pr(\sigma(i_1) = j_1, \ldots , \sigma(i_t) = j_t) &= \Pr(i_1, \ldots i_m \in S, i_{m+1}, \ldots i_t \notin S)\cdot\\
& \cdot \Pr(\sigma(i_1) = j_1, \ldots, \sigma(i_m) = j_m |i_1, \ldots i_m \in S, i_{m+1}, \ldots i_t \notin S)\cdot\\
& \cdot \Pr (\sigma(i_{m+1}) = j_{m+1}, \ldots , \sigma(i_t) = j_t |
i_1, \ldots i_m \in S, i_{m+1}, \ldots i_t \notin S)
\end{align*}
and each term on the right hand side takes the same value whether
$S, \tau_1$, and $\tau_2$ are chosen independently and uniformly at
random, or $S$ is chosen uniformly at random and $\tau_1$ and
$\tau_2$ are chosen independently and uniformly from $t$-wise
uniform sets.

This observation gives us a naive approach to constructing $t$-wise uniform sets recursively. Letting $\per(n,t)$ denote the minimal size of a $t$-wise uniform set of permutations on $n$ elements, we obtain the recursion
\begin{equation}
\label{initial_recursion} \per(2n,t) \leq \binom{2n}{n}\per(n,t)^2 .
\end{equation}
We can use this recursion, together with the initial conditions
$\per(n,t) = n!$ for $t\ge n$, to construct $t$-wise uniform sets;
however, this recursion does not yield non-trivial constructions.

The first main contribution of this work is to propose an improved
divide-and-conquer scheme for creating $t$-wise permutations.
Specifically, adapting an idea proposed in
\cite{BenjaminiGurelGurevichPeled07}, we observe that to construct a
$(2t+1)$-wise uniform set of permutations on $2n$ elements, it
suffices to use a $(2t+1)$-wise uniform set on $n$ elements and a
$t$-wise uniform set on $n$ elements in Step 2 above, instead of two
$(2t+1)$-wise uniform sets (see Figure~\ref{figure}). This yields
the recursion:

\begin{equation}
\label{equation.recursion} \per(2n, 2t+1) \leq \binom{2n}{n} \per(n,
2t+1) \per(n,t),
\end{equation}
where we define $\per(n,2t+1) = \per(n,n) = n!$ when $2t+1 > n$.
Unlike the naive recursion \eqref{initial_recursion}, this recursion
yields a construction of a family of non-trivial $t$-wise uniform
sets.

\begin{theorem}\label{theorem.main} $\per(n,t)\le t^{2n}$ when $n$ and $t$ have the form $n=2^m$ and
$t=2^\ell-1$ for integers $m\ge 1$ and $\ell\ge 2$.
\end{theorem}

While $t^{2n}$ is much smaller than the trivial upper bound $n!$, it
is still much larger than the existence result of
\cite{KuperbergLovettPeled11}. The second main contribution of this
work is to suggest a potential extension that could lead to a
smaller construction: Denoting by $\binom{[n]}{k}$ the set of
subsets of $[n]$ of size exactly $k$, we suggest to replace the
uniformly chosen set of $n$ elements from Step 1 above with a set of
$n$ elements that is $t$-wise uniform, in the following sense.
\begin{definition}
A {\em $(2n,t)$-selection} is a subset $\bigS \subseteq
\binom{[2n]}{n}$ such that for all $I=\{i_1, \ldots, i_t\} \subseteq
[2n]$ and all $J\subseteq I$, the probability that $J\subseteq S$
and $I\setminus J\subseteq S^c$ is the same whether $S$ is chosen
uniformly from $\binom{[2n]}{n}$ or uniformly from $\bigS$.
\end{definition}

A $(2n,t)$-selection is equivalent to a $t-(2n, n, \cdot)$
combinatorial design (see Section~\ref{sec.design}). To obtain a
$t$-wise uniform set, we can choose $S$ uniformly from a
$(2n,t)$-selection in Step~1 above, rather than uniformly from
$\binom{[2n]}{n}$. Letting $\sel(2n,t)$ denote the minimal size of a
$(2n,t)$-selection on $2n$ elements, this allows us to replace the
naive recursion \eqref{initial_recursion} with the following
recursion:
\begin{equation}
\label{initial_recursion2}
\per(2n,t) \leq \sel(2n,t)\per(n,t)^2 ,
\end{equation}
and the more powerful recursion \eqref{equation.recursion} with:
\begin{equation}
\label{equation.recursion2}
\per(2n, 2t+1) \leq \sel(2n,2t+1) \per(n, 2t+1) \per(n,t) .
\end{equation}

These recursions imply that a non-trivial construction of a
$(2n,t)$-selection with the appropriate parameters will result in a
non-trivial construction of a $t$-wise uniform set. For example, it
is shown in \cite{RayChaudhuriWilson75} that a $(2n,t)$ selection
(regarded as a $t-(2n,n,\cdot)$ design) must be of size at least
$\binom{2n}{t/2}$ if $t$ is even and of size at least
$2\binom{2n-1}{(t-1)/2}$ if $t$ is odd. If there were an explicit
construction of a $(2n,t)$ selection of size $n^{ct}$ for some $c>0$
(the existence of such a selection is proven in
\cite{KuperbergLovettPeled11} but no explicit construction is
known), then recursion~\eqref{initial_recursion2} would lead to an
explicit construction of a $t$-wise uniform set on $n$ elements of
size at most $(t+1)^{c'n}$ for some $c'>0$ and
recursion~\eqref{equation.recursion2} would lead to a $t$-wise
uniform set of size at most $2^{c_t(\log{n})^{\log_2{(t+1)}}}$ for
some $c_t>0$ depending only on $t$. Thus our improved recursion
allows us more efficiently to reduce the problem of finding $t$-wise
uniform sets to the problem of finding $(2n,t)$-selections. If a
$(2n,t)$-selection is allowed to be a multiset, then a reduction in
the reverse direction holds as well: a $(2n,t)$-selection $\bigS$
can be obtained from a $t$-wise uniform set $T$ by taking the family
of sets of elements mapped to $1, \ldots , n$ by the elements of
$T$.

Unfortunately, we have been unable to come up with non-trivial
constructions of $(2n,t)$-selections with the appropriate
parameters, or to find such constructions in the literature on
combinatorial designs \cite{ColbournDinitz07, Kreher93}. Some more
suggestions for extending our approach are described in
Section~\ref{sec.future_work}.

\section{The improved recursion}
\label{sec.recursion}

In this section we derive the recursion~\eqref{equation.recursion2}.
Recursion~\eqref{equation.recursion} follows immediately by using
the complete selection.  For convenience, we will call a permutation
chosen uniformly at random from a $t$-wise uniform set a {\em
$t$-wise uniform permutation}.

Let $A$ and $B$ be $t$- and $(2t+1)$-wise uniform sets on $n$
elements, respectively, and let $\bigS$ be a $(2n,2t+1)$-selection.
For each $S\in\bigS$ let $f=f_S$ and $g=g_S$ denote bijections from
$S$ and $S^c$, respectively, to $[n]$. For each $\sigma \in A$,
$\tau \in B$, and $S \in \bigS$ define a permutation
$\mu_{S,\sigma,\tau}$ on $2n$ elements as follows:
\begin{equation}\label{eq:mu_S_sigma_tau_def}
\mu_{S,\sigma,\tau}(x) = \begin{cases}  (\tau \circ \sigma) (f(x)) & x \in S \\
                                        \tau(g(x))+n & x \in S^c
\end{cases}.
\end{equation}
The permutation $\mu_{S,\sigma,\tau}$ sends the elements of $S$ to
$\{1, \ldots , n\}$ and the elements not in $S$ to $\{n+1, \ldots,
2n\}$. The behavior of $\mu_{S,\sigma,\tau}$ on $S$ is determined by
$\tau \circ \sigma$ and the behavior on $S^c$ is determined by
$\tau$ (see Figure~\ref{figure}).


\begin{figure}[h!]
\begin{center}
\includegraphics[width=0.5\textwidth]{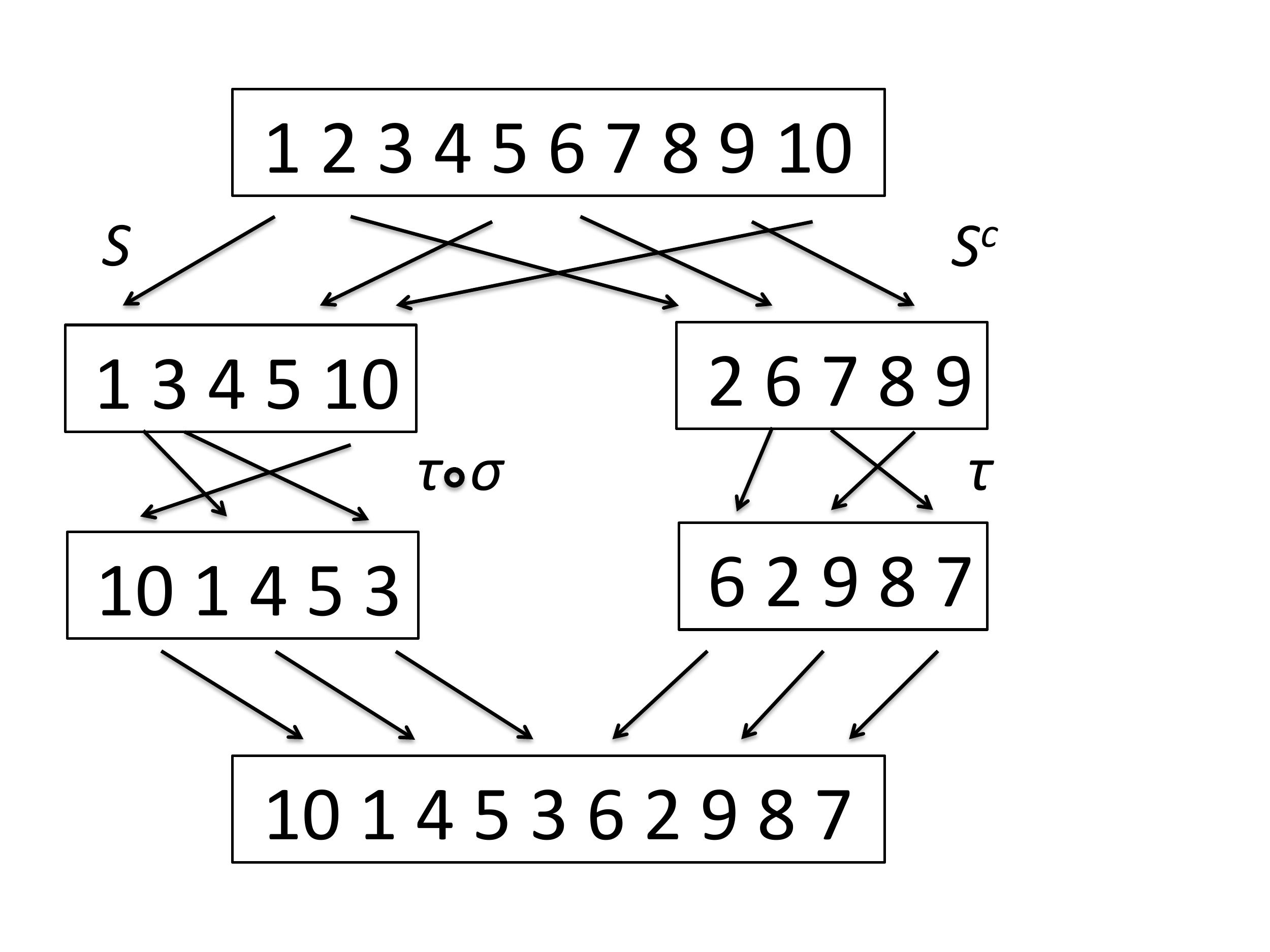}
\caption{\label{figure} A $(2t+1)$-wise permutation $\mu$ on $2n$
elements is constructed from a set $S$ chosen from a
$(2n,2t+1)$-selection, and permutations $\sigma$ and $\tau$ drawn
from $t$- and $(2t+1)$-wise uniform sets of permutations on $n$
elements, respectively. The behavior of $\mu$ on $S$ is determined
by $\tau \circ \sigma$, and the behavior of $\mu$ on $S^c$ is
determined by $\tau.$ Note that in this diagram, we are showing which number gets sent to which position; for example, $\mu(10) = 1$ and $\mu(1) = 2$.}
\end{center}
\end{figure}

\begin{proposition}
\label{prop.main} $\M = \{\mu_{S,\sigma,\tau} : S \in \bigS, \sigma
\in A, \tau \in B\}$ is a $(2t+1)$-wise uniform set.
\end{proposition}
Recursion~\eqref{equation.recursion2} follows from the proposition
since $|\M|\le |\bigS||A||B|$. In the rest of this section we prove
Proposition~\ref{prop.main}.

We start by introducing some notation. For a set of indices $I$, a
function $h$ and a set $S$, let $x_I = y_I$ denote $x_i = y_i$ for
all $i \in I$, let $h(x_I) = y_I$ denote $h(x_i) = y_i$ for all $i
\in I$, let $x_I \in S$ denote $x_i \in S$ for all $i \in I$, and
let $x_I \notin S$ denote $x_i \notin S$ for all $i \in I$.

The key step in our proof of the proposition is the following lemma.
It asserts that if $\sigma$ and $\tau$ are $t$- and $(2t+1)$-wise
uniform permutations on $[n]$, respectively, independent of each
other, then the pair of permutations $\tau\circ\sigma$ and $\tau$,
while neither uniform nor independent in general, behave exactly as
uniform and independent when queried together on at most $2t+1$
inputs.



\begin{lemma}
\label{main_lemma} Let $\sigma$ and $\tau$ be $t$- and $(2t+1)$-wise
uniform permutations on $[n]$, respectively, independent of each
other. For any $r,s\ge0$ satisfying $r + s = 2t+1$ and any distinct
$i_1, \ldots , i_r$, distinct $j_1, \ldots, j_r$, distinct $k_1,
\ldots , k_s$, and distinct $\ell_1, \ldots , \ell_s$ in $[n]$, we
have
\begin{equation}\label{eq:lemma_equality}
\Pr\left((\tau \circ \sigma)(i_{[r]}) = j_{[r]}, \tau(k_{[s]}) =
\ell_{[s]}\right) = \frac{(n-r)!(n-s)!}{n!^2}.
\end{equation}
\end{lemma}

\begin{proof}
Fix $r$, $s$ and sets of indices as in the lemma. Let
$$M = \{m_1\in[r]\, : \mbox{there exists a $m_2\in[s]$ such that $j_{m_1} = \ell_{m_2} $}\}.$$
For ease of notation, reorder the indices so that if $j_{m_1} =
\ell_{m_2}$, then $m_1 = m_2$. Observe that on the event that $(\tau
\circ \sigma)(i_{[r]}) = j_{[r]}$ and $\tau(k_{[s]}) = \ell_{[s]}$
we must also have that $\sigma(i_M) = k_M$ and $\sigma(i_b) \neq
k_c$ for any $b \in [r] \setminus M$ and $c \in [s] \setminus M$.
Let $R$ denote the set of all permutations for which these two
conditions hold; i.e.
$$R = \{\alpha\in S_n : \alpha(i_M) = k_M \mbox{ and } \alpha(i_b) \neq k_c \mbox{ for all } b \in [r] \setminus M\text{ and }c \in [s] \setminus M \}.$$
It follows that
$$\Pr\left((\tau \circ \sigma)(i_{[r]}) = j_{[r]}, \tau(k_{[s]}) = \ell_{[s]}\right) = \Pr\left((\tau \circ \sigma)(i_{[r]}) = j_{[r]}, \tau(k_{[s]}) = \ell_{[s]} \mbox{ and } \sigma \in R \right).$$
Breaking the event on the right-hand side into disjoint events based
on the value $\sigma$ takes, and then using the independence of
$\tau$ and $\sigma$, we have
\begin{align}
\Pr\left((\tau \circ \sigma)(i_{[r]}) = j_{[r]}, \tau(k_{[s]}) =
\ell_{[s]}\right)&=\sum_{\alpha \in R} \Pr\left((\tau \circ \alpha)(i_{[r]}) = j_{[r]}, \tau(k_{[s]}) = \ell_{[s]} \mbox{ and } \sigma = \alpha \right)\nonumber\\
&=\sum_{\alpha \in R} \Pr\left((\tau \circ \alpha)(i_{[r]}) =
j_{[r]}, \tau(k_{[s]}) = \ell_{[s]}\right)\Pr(\sigma =
\alpha).\label{eq:tau_sigma_prob_with_R_decomp}
\end{align}
The definition of $R$ shows that for any $\alpha \in R$ we have
$$\left\{ (\tau \circ \alpha)(i_{[r]}) = j_{[r]}, \tau(k_{[s]}) = \ell_{[s]} \right\} =  \left\{ (\tau \circ \alpha)(i_{[r] \setminus M}) = j_{[r] \setminus M}, \tau(k_{[s]}) =
\ell_{[s]}\right\}.$$ Moreover, by the definitions of $R$ and $M$,
we know that the elements of $\alpha(i_{[r] \setminus M}) \cup
k_{[s]}$ are all distinct and that the elements of  $j_{[r]
\setminus M} \cup \ell_{[s]}$ are distinct. Thus, recalling that
$r+s=2t+1$ we can use the $(2t+1)$-wise uniformity of $\tau$ and
equation~\eqref{t-wise_def} to write
\begin{align*}
\Pr\left((\tau \circ \alpha)(i_{[r]}) = j_{[r]}, \tau(k_{[s]}) = \ell_{[s]}\right)  &= \frac{1}{n\cdot(n-1)\cdots(n-(r+s-|M|) +1)} =\\
&= \frac{(n-r-s+|M|)!}{n!}.
\end{align*}
Hence we may continue \eqref{eq:tau_sigma_prob_with_R_decomp}
and obtain
\begin{align}
\Pr\big((\tau \circ \sigma)(i_{[r]}) &= j_{[r]}, \tau(k_{[s]}) =
\ell_{[s]}\big) =\nonumber\\
&=\sum_{\alpha \in R}\frac{(n-r-s+|M|)!}{n!}\Pr(\sigma =
\alpha)=\frac{(n-r-s+|M|)!}{n!}\Pr(\sigma \in
R).\label{eq:tau_sigma_prob_R}
\end{align}
Now let $\sigma'$ be a permutation chosen uniformly at random from
$S_n$. We claim that
\begin{equation}\label{eq:sigma_sigma_prime_R_eq}
  \Pr(\sigma\in R) = \Pr(\sigma'\in R).
\end{equation}
To see this, we consider separately the cases $r \leq t$ and $s \leq
t$. One of these cases must hold since $r+s=2t+1$ by assumption.
First, suppose $r \leq t$. We partition the event $\sigma \in R$
into disjoint events based on the values $\sigma$ assigns to $i_1,
\ldots, i_r$, as follows.
\begin{equation}\label{eq:sigma_R_decomp}
\Pr(\sigma \in R) = \sum \Pr(\sigma(i_{[r]}) = x_{[r]}),
\end{equation}
where the sum is taken over all sets of distinct $x_1, \ldots , x_r
\in[n]$ such that $x_M = k_M$ and $x_b \neq k_c \mbox{ for all } b
\in [r] \setminus M$ and $c \in [s] \setminus M $. Now, since $r
\leq t$ and $\sigma$ is $t$-wise uniform, for each fixed
$x_1,\ldots, x_r$ we have
\begin{equation}\label{eq:sigma_sigma_prime_eq}
\Pr(\sigma(i_{[r]}) = x_{[r]}) =  \Pr(\sigma'(i_{[r]}) = x_{[r]}).
\end{equation}
Combining \eqref{eq:sigma_R_decomp} and
\eqref{eq:sigma_sigma_prime_eq} we obtain
$$\Pr(\sigma \in R) = \sum \Pr(\sigma(i_{[r]}) = x_{[r]}) = \sum \Pr(\sigma'(x_{[s]}) = k_{[s]})  = \Pr(\sigma' \in R),$$
where the sums are over the same choices of $x_1,\ldots, x_r$ as in
\eqref{eq:sigma_R_decomp}.

The case $s \leq t$ proceeds similarly, by partitioning the event
$\sigma\in R$ according to the inverse images of $k_1,\ldots, k_s$
through $\sigma$. We obtain
$$\Pr(\sigma \in R) = \sum \Pr(\sigma(x_{[s]}) = k_{[s]}) = \sum \Pr(\sigma'(x_{[s]}) = k_{[s]})  = \Pr(\sigma' \in R),$$
where the sum is taken over all sets of distinct $x_1, \ldots , x_s
\in [n]$ such that $x_M = i_M$ and $x_c \neq i_b$ for all $ b \in
[r] \setminus M$ and $c \in [s] \setminus M $. The second equality
follows because $s \leq t$ and $\sigma$ is $t$-wise uniform. Thus we
have established \eqref{eq:sigma_sigma_prime_R_eq} in all cases.

Finally, for the uniformly random permutation $\sigma'$ it is
straightforward to verify that
\begin{equation*}
  \Pr(\sigma'\in R) = \frac{(n-r)!(n-s)!}{(n-r-s+|M|)!n!}.
\end{equation*}
Thus the lemma follows from \eqref{eq:sigma_sigma_prime_R_eq} and
\eqref{eq:tau_sigma_prob_R}.
\end{proof}

\begin{proof}[Proof of Proposition~\ref{prop.main}:]
To prove Proposition~\ref{prop.main}, we need to show that if $\mu$
is chosen uniformly at random from $\M$ and $i_1, \ldots ,
i_{2t+1}\in[2n]$ are distinct indices and $j_1, \ldots ,
j_{2t+1}\in[2n]$ are distinct indices then
\begin{equation}\label{eq:prop_1_goal}
\Pr(\mu(i_{[2t+1]}) = j_{[2t+1]}) = \frac{(2n-(2t+1))!}{(2n)!}.
\end{equation}
Fix two sets of distinct indices, $i_1, \ldots , i_{2t+1}\in[2n]$
and $j_1, \ldots , j_{2t+1}\in[2n]$. By reordering the indices, we
may assume without loss of generality that $1 \leq j_1, \ldots , j_m
\leq n$ and $n+1 \leq j_{m+1}, \ldots, j_{2t+1} \leq 2n$ for some
$0\le m\le 2t+1$. Observing that $(S_1, \sigma_1, \tau_1) \neq (S_2,
\sigma_2, \tau_2)$ implies that $\mu_{S_1, \sigma_1, \tau_1} \neq
\mu_{S_2, \sigma_2, \tau_2}$ we conclude that choosing an element
$\mu$ uniformly at random from $\M$ is equivalent to choosing
elements $S$, $\sigma$, and $\tau$ uniformly at random and
independently from $\bigS$, $A$, and $B$, respectively, and letting
$\mu = \mu_{S, \sigma,\tau}$. Defining $S,\sigma,\tau$ and $\mu$ in
this way, we have
\begin{align*}
&\Pr(\mu(i_{[2t+1]}) = j_{[2t+1]}) = \nonumber\\
&= \Pr(i_{[m]} \in S, i_{[2t+1]\setminus[m]} \notin S,
(\tau\circ\sigma\circ f)(i_{[m]}) = j_{[m]}, (\tau\circ
g)(i_{[2t+1]\setminus[m]}) + n =
j_{[2t+1]\setminus[m]}),
%
\end{align*}
where the equality $(\tau\circ g)(i_{[2t+1]\setminus[m]}) + n =
j_{[2t+1]\setminus[m]}$ should be interpreted as $(\tau\circ g)(i_k)
+ n = j_{k}$ for all $k\in[2t+1]\setminus[m]$. Conditioning on $S$,
we obtain
\begin{align}
&\Pr(\mu(i_{[2t+1]}) = j_{[2t+1]}) = \nonumber\\
&= \E({\bf 1}_{(i_{[m]} \in S,\; i_{[2t+1]\setminus[m]} \notin S)}
\cdot\Pr((\tau\circ\sigma\circ f)(i_{[m]}) = j_{[m]}, (\tau\circ
g)(i_{[2t+1]\setminus[m]}) + n =
j_{[2t+1]\setminus[m]}\,|\,S)),\label{eq:prop_1_reduction_with_conditioning}
\end{align}
where ${\bf 1}_A$ denotes the indicator random variable of the event
$A$. Now, recalling that $\tau$ and $\sigma$ are independent of $S$,
and that $f$ and $g$ depend only on $S$, we may apply
Lemma~\ref{main_lemma} to conclude that for every $S$ satisfying
$i_{[m]} \in S$ and $i_{[2t+1]\setminus[m]} \notin S$ we have
\begin{equation*}
  \Pr((\tau\circ\sigma\circ f)(i_{[m]}) = j_{[m]}, (\tau\circ
g)(i_{[2t+1]\setminus[m]}) + n = j_{[2t+1]\setminus[m]}\,|\,S) =
\frac{(n-m)!(n-(2t+1-m))!}{n!^2}.
\end{equation*}
Substituting back into \eqref{eq:prop_1_reduction_with_conditioning}
yields
\begin{equation}\label{eq:before_S_substitution}
  \Pr(\mu(i_{[2t+1]}) = j_{[2t+1]}) =
  \frac{(n-m)!(n-(2t+1-m))!}{n!^2}\Pr(i_{[m]} \in S, i_{[2t+1]\setminus[m]} \notin
  S).
\end{equation}
Since $\bigS$ is a $(2n, 2t+1)$-selection we have
\begin{equation*}
  \Pr(i_{[m]} \in S, i_{[2t+1]\setminus[m]} \notin
  S) = \frac{\binom{2n-(2t+1)}{n-m}}{\binom{2n}{n}}.
\end{equation*}
Substituting this into \eqref{eq:before_S_substitution} yields
\begin{equation*}
  \Pr(\mu(i_{[2t+1]}) = j_{[2t+1]}) = \frac{(2n-(2t+1))!}{(2n)!},
\end{equation*}
and the proposition follows.\qedhere

\end{proof}

Our construction is a modification of the method used in
\cite{BenjaminiGurelGurevichPeled07} for creating $t$-wise
independent strings. In the context of that work, a (binary,
unbiased) $t$-wise independent string is a random vector
$X=(X_1,\ldots, X_n)\in\{0,1\}^n$ satisfying that $\Pr(X_i=1)=1/2$
for every $i$ and that $(X_{i_1},\ldots, X_{i_t})$ are independent
for every $1\le i_1<\cdots<i_t\le n$. It was shown there that if $X$
is a $(2t+1)$-wise independent string and $Y$ is a $t$-wise
independent string (both of the same length) such that $X$ and $Y$
are independent then the string formed from the concatenation of $X$
and $X+Y$ is $(2t+1)$-wise independent. The proof there is similar
to our own, but made easier by the fact that the group $\{0,1\}^n$
is simpler than the group $S_n$ and by the fact that no selection is
necessary in the context of $t$-wise independent strings. It is of
interest to understand the extent to which this method is applicable
and find a common generalization of the above two scenarios.

\section{The construction}
\label{sec.construction}

In this section we use recursion~\eqref{equation.recursion} to construct an infinite family of non-trivial $t$-wise uniform sets, as stated in Theorem~\ref{theorem.main}.

\begin{proof}[Proof of Theorem~\ref{theorem.main}:]
We start by noting a few simple facts. First, we trivially have that
$\per(n,t)=n!\le t^{2n}$ when $t\geq n$. Second, we observe that the
set of permutations $\{\sigma_b\}$, for $0\le b<n$, defined by
$\sigma_b(i) = i+b\pmod n$ is a $1$-wise uniform set of permutations
on $\{0, \ldots , n-1\}$. Thus, $\per(n,1)\le n$ for all $n$. (In
fact, since $\per(n,t) \geq n(n-1)\cdots (n-t+1)$ by
\eqref{t-wise_def}, this is an equality.)

Next, we establish the theorem for $t=3$ (or $\ell=2$). That is, we
show that $\per(2^m,3)\le 3^{2^{m+1}}$ for all integer $m\ge 1$. The
proof is by induction. For $m=1$, we have $\per(2,3)=2\le 3^4$ as
required. For $m>1$, we have by equation~\eqref{equation.recursion},
the above observations and the induction hypothesis that
\begin{equation*}
  \per(2^m,3)\le 2^{2^m}\per(2^{m-1},1)\per(2^{m-1},3)\le 2^{2^m}2^{m-1}3^{2^m} = \frac{2^{2^m+m-1}}{3^{2^m}}3^{2^{m+1}}.
\end{equation*}
It follows that $\per(2^m,3)\le 3^{2^{m+1}}$, as required, by using
that $2^{(m-1)}\le (3/2)^{2^m}$ for all $m\ge 2$ since
$\log_2(3/2)\ge 1/2$ and $x\le 2^x$ for $x\ge 1$.

Finally, we establish the theorem in general. We will prove by
induction on $m$ that for each $m\ge 1$, the claim holds for all
$\ell\ge 2$. Again, the case $m=1$ follows by our observation that
$\per(n,t)=n! \leq t^{2n}$ when $t\ge n$. Suppose that $m>1$ and fix
$\ell\ge 2$. We have already established the case $\ell=2$ and the
case $t \ge n$, so we may assume that $\ell\ge 3$ and satisfies
$2^\ell-1<2^m$. We also assume the induction hypothesis
$$\per(2^{m-1},2^{\ell'} -1) \leq (2^{\ell'}-1)^{2^m}$$
for all $\ell' \ge 2$. Thus, by~\eqref{equation.recursion},
\begin{align*}
  \per(2^m,2^\ell-1)&\le 2^{2^m} \per(2^{m-1},2^{\ell-1}-1)\per(2^{m-1},2^\ell-1)\le\\
  &\le 2^{2^m}(2^{\ell-1}-1)^{2^m}(2^\ell-1)^{2^m}\le
  (2^{\ell}-1)^{2^{m+1}},
\end{align*}
as required.
\end{proof}

\section{The connection with $t$-designs}
\label{sec.design} To obtain a smaller construction of a $t$-wise
uniform set of permutations, it would suffice to construct a small
$(2n,t)$-selection that could be used in
recursions~\eqref{initial_recursion2} or
~\eqref{equation.recursion2}. Selections are a special case of
$t$-designs, defined as follows (see for example
\cite{ColbournDinitz07}):
\begin{definition}
  A {\em $t-(v,k,\lambda)$ design} is a subset $\bigS \subseteq \binom{[v]}{k}$  satisfying that for every distinct $i_1,\ldots, i_t\in [v]$ we  have
  \begin{equation*}
    |\{S\in \bigS\ :\ i_1,\ldots, i_t\in S\}| = \lambda.
  \end{equation*}
\end{definition}
Thus, a $(2n,t)$- selection is a $t-(2n,n, \lambda)$ design for some $\lambda$. The following simple lemma shows that the converse is also true (see for example \cite[Proposition 1]{RayChaudhuriWilson75}):
\begin{lemma}\label{lemma.selection}
Let $\bigS$ be a $t-(v,k,\lambda)$ design, $i_1, \ldots, i_t \in
[v]$, and $m \leq t$. Then the probability that $i_1, \ldots , i_m
\in S$ and $i_{m+1}, \ldots , i_t \notin S$ is the same whether $S$
is chosen uniformly from $\binom{[v]}{k}$ or uniformly from $\bigS$.
\end{lemma}
Thus, to improve the parameters of our construction using
recursions~\eqref{initial_recursion2} or
\eqref{equation.recursion2}, we need efficient $t-(2n,n, \lambda)$
designs for large $n$ (and $t\ge 4$). However, we have not been able
to find such designs in the literature.


\section{Directions for future work}\label{sec.future_work}

The parameters of our result could possibly be improved by
considering a generalization of our construction that is based on an
idea from \cite{BenjaminiGurelGurevichPeled07}. The generalization
is as follows: for any $k\ge 2$, a $(2t+1)$-wise uniform permutation
on $kn$ elements may be created from a uniformly random partition
$T$ of $[kn]$ into $k$ groups of size $n$, and random permutations
$\tau, \sigma_1,\ldots, \sigma_{k-1}$ where $\tau$ is a
$(2t+1)$-wise uniform permutation on $n$ elements and each
$\sigma_i$ is a $t$-wise uniform permutation on $n$ elements, and
$T,\tau,\sigma_1,\ldots,\sigma_{k-1}$ are independent. The
permutation is formed by partitioning the $kn$ inputs to $k$ groups
according to $T$, applying $\tau$ to the first group and applying
$\tau\circ\sigma_i$ to the $(i+1)$'st group for $1\le i\le k-1$. The
fact that the resulting permutation $\mu$ is $(2t+1)$-wise uniform
follows easily from Lemma~\ref{main_lemma} and the simple
observation that given any inputs $i_1,\ldots, i_{2t+1}$ to $\mu$,
and any choice of $T$, there can be at most one $\sigma_i$ which
determines the behavior of $\mu$ on more than $t$ of these inputs. A
significant advantage of this generalized construction is that
although the number $k$ can be arbitrarily large, still only one of
the permutations used needs to be $(2t+1)$-wise uniform. Moreover,
as in recursion \eqref{equation.recursion2}, we may relax the
condition that $T$ be uniformly random to the condition that $T$ be
a $(2t+1)$-wise partition in an appropriate sense (though for $k\ge
3$ this notion is no longer connected with $t$-designs). We were
unable to obtain any improvement in our parameters from this
generalized construction, but we see it as a potential starting
point for future improvements on our construction.

The construction in \cite{BenjaminiGurelGurevichPeled07} was
inspired by the $(u|u+v)$ method for combining error-correcting
codes \cite[p. 76]{MacWilliamsSloane}. Another promising direction
for future research is to check whether other methods for combining
codes can be adapted to yield constructions of $t$-wise uniform sets
of permutations.

\bibliographystyle{plain}
\bibliography{designs}

\end{document}